\documentclass[reqno, 12pt]{amsart}
\pdfoutput=1
\makeatletter
\let\origsection=\section \def\section{\@ifstar{\origsection*}{\mysection}} 
\def\mysection{\@startsection{section}{1}\z@{.7\linespacing\@plus\linespacing}{.5\linespacing}{\normalfont\scshape\centering\S}}
\makeatother        

\usepackage{amsmath,amssymb,amsthm}
\usepackage{mathrsfs}
\usepackage{mathabx}\changenotsign
\usepackage{dsfont}
 
\usepackage{xcolor}
\usepackage[backref]{hyperref}
\hypersetup{
    colorlinks,
    linkcolor={red!60!black},
    citecolor={green!60!black},
    urlcolor={blue!60!black}
}

\usepackage[open,openlevel=2,atend]{bookmark}

\usepackage[abbrev,msc-links,backrefs]{amsrefs} 
\usepackage{doi}

\renewcommand{\PrintDOI}[1]{\doi{#1}}

\usepackage[T1]{fontenc}
\usepackage{lmodern}
\usepackage[babel]{microtype}
\usepackage[english]{babel}

\usepackage{geometry}
\numberwithin{equation}{section}
\numberwithin{figure}{section}

\usepackage{enumitem}
\def\rmlabel{\upshape({\itshape \roman*\,})}

\def\alabel{\upshape({\itshape \alph*\,})}

\let\polishlcross=\l
\def\l{\ifmmode\ell\else\polishlcross\fi}

\let\sm=\setminus

\makeatletter
\def\moverlay{\mathpalette\mov@rlay}
\def\mov@rlay#1#2{\leavevmode\vtop{   \baselineskip\z@skip \lineskiplimit-\maxdimen
   \ialign{\hfil$\m@th#1##$\hfil\cr#2\crcr}}}
\newcommand{\charfusion}[3][\mathord]{
    #1{\ifx#1\mathop\vphantom{#2}\fi
        \mathpalette\mov@rlay{#2\cr#3}
      }
    \ifx#1\mathop\expandafter\displaylimits\fi}
\makeatother

\newcommand{\dcup}{\charfusion[\mathbin]{\cup}{\cdot}}

\DeclareFontFamily{U}  {MnSymbolC}{}
\DeclareSymbolFont{MnSyC}         {U}  {MnSymbolC}{m}{n}
\DeclareFontShape{U}{MnSymbolC}{m}{n}{
    <-6>  MnSymbolC5
   <6-7>  MnSymbolC6
   <7-8>  MnSymbolC7
   <8-9>  MnSymbolC8
   <9-10> MnSymbolC9
  <10-12> MnSymbolC10
  <12->   MnSymbolC12}{}
\DeclareMathSymbol{\powerset}{\mathord}{MnSyC}{180}

\usepackage{tikz}
\usetikzlibrary{calc,decorations.pathmorphing,through}
\usetikzlibrary{shapes,decorations}
\usetikzlibrary{graphs,graphs.standard}
\usetikzlibrary{decorations.markings}
\usetikzlibrary{positioning, intersections, backgrounds}

\newcommand{\triple}[7]{

	\ifx\relax#4\relax
		\def\qoffs{0pt}
	\else
		\def\qoffs{#4}
	\fi

	\def\qhedge{
		($#1+#3!\qoffs!-90:#2-#3$) --
		($#2+#1!\qoffs!-90:#3-#1$) --
		($#3+#2!\qoffs!-90:#1-#2$) -- cycle}

	\coordinate (12) at ($#1!\qoffs!90:#2$);
	\coordinate (13) at ($#1!\qoffs!-90:#3$);
	\coordinate (23) at ($#2!\qoffs!90:#3$);
	\coordinate (21) at ($#2!\qoffs!-90:#1$);
	\coordinate (31) at ($#3!\qoffs!90:#1$);
	\coordinate (32) at ($#3!\qoffs!-90:#2$);
	
	\def\nqhedge{
		(13) let \p1=($(13)-#1$), \p2=($(12)-#1$) in
			arc[start angle={atan2(\y1,\x1)}, delta angle={atan2(\y2,\x2)-atan2(\y1,\x1)-360*(atan2(\y2,\x2)-atan2(\y1,\x1)>0)}, x radius=\qoffs, y radius=\qoffs] --
		(21) let \p1=($(21)-#2$), \p2=($(23)-#2$) in
			arc[start angle={atan2(\y1,\x1)}, delta angle={atan2(\y2,\x2)-atan2(\y1,\x1)-360*(atan2(\y2,\x2)-atan2(\y1,\x1)>0)}, x radius=\qoffs, y radius=\qoffs] --
		(32) let \p1=($(32)-#3$), \p2=($(31)-#3$) in
			arc[start angle={atan2(\y1,\x1)}, delta angle={atan2(\y2,\x2)-atan2(\y1,\x1)-360*(atan2(\y2,\x2)-atan2(\y1,\x1)>0)}, x radius=\qoffs, y radius=\qoffs] --
		cycle}

		\ifx\relax#5\relax
		\def\qlwidth{1pt}
	\else
		\def\qlwidth{#5}
	\fi
	
		\ifx\relax#7\relax
		\fill \nqhedge;
	\else
		\fill[#7]\nqhedge;
	\fi

		\ifx\relax#6\relax
		\draw[line width=\qlwidth,rounded corners=\qoffs]\nqhedge;
	\else
		\draw[line width=\qlwidth,#6]\nqhedge;
	\fi
}

\newcommand{\arcThroughThreePoints}[4][]{
\coordinate (middle1) at ($(#2)!.5!(#3)$);
\coordinate (middle2) at ($(#3)!.5!(#4)$);
\coordinate (aux1) at ($(middle1)!1!90:(#3)$);
\coordinate (aux2) at ($(middle2)!1!90:(#4)$);
\coordinate (center) at ($(intersection of middle1--aux1 and middle2--aux2)$);
\draw[#1] 
 let \p1=($(#2)-(center)$),
      \p2=($(#4)-(center)$),
      \n0={veclen(\p1)},      
      \n1={atan2(\y1,\x1)}, 
      \n2={atan2(\y2,\x2)},
      \n3={\n2>\n1?\n2:\n2+360}
    in (#2) arc(\n1:\n3:\n0);
}

\newcommand{\quadrupel}[8]{

		\ifx\relax#5\relax
		\def\qoffs{0pt}
	\else
		\def\qoffs{#5}
	\fi

				\def\qhedge{
		($#1+#4!\qoffs!-90:#2-#4$) -- 
		($#2+#1!\qoffs!-90:#3-#1$) -- 
		($#3+#2!\qoffs!-90:#4-#2$) -- 
		($#4+#3!\qoffs!-90:#1-#3$) -- cycle}

	\coordinate (12) at ($#1!\qoffs!90:#2$);
	\coordinate (14) at ($#1!\qoffs!-90:#4$);
	\coordinate (23) at ($#2!\qoffs!90:#3$);
	\coordinate (21) at ($#2!\qoffs!-90:#1$);
	\coordinate (34) at ($#3!\qoffs!90:#4$);
	\coordinate (32) at ($#3!\qoffs!-90:#2$);
	\coordinate (41) at ($#4!\qoffs!90:#1$);
	\coordinate (43) at ($#4!\qoffs!-90:#3$);
	
	\def\nqhedge{
		(14) let \p1=($(14)-#1$), \p2=($(12)-#1$) in 
			arc[start angle={atan2(\y1,\x1)}, delta angle={atan2(\y2,\x2)-atan2(\y1,\x1)-360*(atan2(\y2,\x2)-atan2(\y1,\x1)>0)}, x radius=\qoffs, y radius=\qoffs] --
		(21) let \p1=($(21)-#2$), \p2=($(23)-#2$) in 
			arc[start angle={atan2(\y1,\x1)}, delta angle={atan2(\y2,\x2)-atan2(\y1,\x1)-360*(atan2(\y2,\x2)-atan2(\y1,\x1)>0)}, x radius=\qoffs, y radius=\qoffs] --
		(32) let \p1=($(32)-#3$), \p2=($(34)-#3$) in 
			arc[start angle={atan2(\y1,\x1)}, delta angle={atan2(\y2,\x2)-atan2(\y1,\x1)-360*(atan2(\y2,\x2)-atan2(\y1,\x1)>0)}, x radius=\qoffs, y radius=\qoffs] --
		(43) let \p1=($(43)-#4$), \p2=($(41)-#4$) in 
			arc[start angle={atan2(\y1,\x1)}, delta angle={atan2(\y2,\x2)-atan2(\y1,\x1)-360*(atan2(\y2,\x2)-atan2(\y1,\x1)>0)}, x radius=\qoffs, y radius=\qoffs] --
		cycle}

		\ifx\relax#6\relax
		\def\qlwidth{1pt}
	\else
		\def\qlwidth{#6}
	\fi
	
		\ifx\relax#8\relax
		\fill \nqhedge;
	\else
		\fill[#8]\nqhedge;
	\fi

		\ifx\relax#7\relax
		\draw[line width=\qlwidth,rounded corners=\qoffs]\nqhedge;
	\else
		\draw[line width=\qlwidth,#7]\nqhedge;
	\fi
}

\usepackage{graphicx}
\usepackage{wrapfig}

\usepackage{subcaption}
\let\epsilon=\varepsilon

\let\rho=\varrho
\let\theta=\vartheta
\let\kappa=\varkappa

\def\ex{\mathrm{ex}}

\newcommand{\cF}{\mathscr{F}}
\newcommand{\cP}{\mathscr{P}}

\theoremstyle{plain}
\newtheorem{thm}{Theorem}[section]
\newtheorem{fact}[thm]{Fact}
\newtheorem{prop}[thm]{Proposition}
\newtheorem{clm}[thm]{Claim}
\newtheorem{cor}[thm]{Corollary}
\newtheorem{lem}[thm]{Lemma}

\theoremstyle{definition}
\newtheorem{dfn}[thm]{Definition}

\usepackage{accents}
\newcommand{\seq}[1]{\accentset{\rightharpoonup}{#1}}

\let\phi=\varphi

\begin{document}

\title[Tur\'an's Theorem for the Fano plane]
{Tur\'an's Theorem for the Fano plane}

\author[Louis Bellmann]{Louis Bellmann} 
\author[Christian Reiher]{Christian Reiher}
\thanks{The second author was supported by the European Research Council 
(ERC grant PEPCo 724903).}
\address{Fachbereich Mathematik, Universit\"at Hamburg, Hamburg, Germany}
\email{Christian.Reiher@uni-hamburg.de}
\email{louisnbell@aol.com}


\begin{abstract} 
Confirming a conjecture of Vera T.~S\'os in a very strong sense, we 
give a complete solution to Tur\'an's hypergraph problem for the Fano plane.
That is we prove for $n\ge 8$ that among all $3$-uniform hypergraphs on $n$ vertices not 
containing the Fano plane there is indeed exactly one whose number of edges 
is maximal, namely the balanced, complete, bipartite hypergraph. 
Moreover, for $n=7$ there is exactly one other extremal configuration with the same 
number of edges: the hypergraph arising from a clique of order $7$  by removing all 
five edges containing a fixed pair of vertices. 

For sufficiently large values $n$ this was proved earlier  
by F\"uredi and Simonovits, and by Keevash and Sudakov, who utilised the stability method.     
\end{abstract}
\maketitle
\section{Introduction}

With his seminal work~\cite{Turan}, Tur\'an initiated extremal graph theory 
as a separate subarea of combinatorics. After proving his well known extremal result 
concerning graphs not containing a clique of fixed order, he proposed to study similar problems 
for graphs arising from platonic solids and for hypergraphs. 
For instance, given a $3$-uniform hypergraph~$F$ and a 
natural number $n$, there arises the question to determine the largest number~$\ex(n, F)$ 
of edges that a $3$-uniform hypergraph~$H$ can have without containing $F$ as a subhypergraph. 

Here a {\it $3$-uniform hypergraph} $H=(V, E)$ consists of a set $V$ of {\it vertices} and 
a collection~${E\subseteq V^{(3)}}=\{e\subseteq V\colon |e|=3\}$ of $3$-element subsets of $V$, that are called the 
{\it edges} of $H$. Since all hypergraphs occurring in this article are $3$-uniform, 
we will henceforth abbreviate the terminology and just say ``hypergraph'' when we mean 
``$3$-uniform hypergraph.''

Despite tremendous efforts over the last 70 years, our knowledge about these  Tur\'an 
functions $n\longmapsto \ex(n, F)$ is very limited, even for very innocent looking 
hypergraphs $F$ such as the tetrahedron $F=K^{(3)}_4$.
It is thus customary to focus on the {\it Tur\'an densities} 
\[
	\pi(F)=\lim_{n\to\infty} \frac{\ex(n, F)}{\binom{n}{3}}\,,
\]
the existence of which follows from the fact that the sequences 
$n\longmapsto \ex(n, F)\big/ \binom{n}{3}$ are, by a result of Katona, Nemetz, and 
Simonovits~\cite{KNS64}, monotonically decreasing. These Tur\'an densities are 
not understood very well either and all one knows in this regard about the tetrahedron 
are the estimates   
\begin{equation} \label{eq:Raz}
	\tfrac59 \le \pi\bigl(K^{(3)}_4\bigr)\le 0.5616\,.
\end{equation}

The lower bound follows from an explicit construction due to Tur\'an himself (see 
e.g.~\cite{Er77}), which is widely believed to be optimal. As observed by Brown~\cite{Br83} 
and Kostochka~\cite{Ko82} there is for each fixed $n$ a large number of $K^{(3)}_4$-free
hypergraphs with the same number of edges that is conjecturally extremal. 
It is often speculated that this non-uniqueness of the extremal configuration is responsible 
for the enormous difficulty of the problem.
The upper bound in~\eqref{eq:Raz} was established
by Razborov~\cite{Ra10} by means of his flag algebraic approach introduced in~\cite{Ra07}.  

Vera T. S\'os proposed to study Tur\'an's hypergraph problem in the special 
case where $F=\cF$ is the {\it Fano plane}, i.e., the projective plane over the field with 
two elements.
More precisely, one takes $\cF$ to be the hypergraph with $7$ vertices, which are the  
{\it points} of the Fano plane, and whose $7$ edges correspond to the {\it lines} of the 
Fano plane (see Fig.~\ref{fig:Fano-1}).  

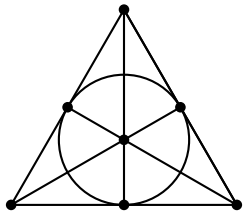
\begin{figure}[ht]
\centering
	\begin{tikzpicture}
		\coordinate(z) at (0, 0);
		\coordinate(a) at (1.5, 0);
		\coordinate(b) at (3, 0);
		\coordinate(x) at (1.5, {1.5*sqrt(3)});
		\coordinate(y) at (3/4, {sqrt(27)/4});
		\coordinate(c) at (3/2, {sqrt(3)/2});
		\coordinate(d) at (9/4, {sqrt(27)/4});
		
		\path[black,thick]
			(x) edge (z)
			(x) edge (b)
			(x) edge (b)
			(x) edge (a)
			(y) edge (b)
			(z) edge (b)
			(z) edge (d);
		\node [draw,black, thick, circle through=(a)] at (c) {};
		
		\foreach \i in {a, b, c, d, x, y, z}
			\fill  (\i) circle (2pt);			
		
	\end{tikzpicture}
        \caption{Fano plane}
  	\label{fig:Fano-1}
\end{figure}	
     
One verifies easily that no matter how the vertices of the Fano plane get coloured 
with two colours, there will always be a monochromatic edge; this fact suggests that  
bipartite hypergraphs could be relevant to the problem under discussion. Given a natural
number~$n$, we denote the {\it balanced, complete, bipartite} hypergraph on $n$ vertices 
by $B_n$.
This hypergraph is defined so as to have a partition $V(B_n)=X\dcup Y$ of its $n$-element 
vertex set with~${\big||X|-|Y|\big|\le 1}$ such that a triple $e\subseteq V(B_n)$ forms an edge 
of $B_n$ if and only if it intersects both $X$ and $Y$. The above observation on vertex 
colourings implies $\cF\not\subseteq B_n$ and, hence, that $\ex(n, \cF)\ge b(n)$, where 
\[
	b(n)=\binom{n}{3}
				-\binom{\lfloor n/2 \rfloor}{3}
				-\binom{\lfloor (n+1)/2 \rfloor}{3} 
\]
denotes the number of edges of $B_n$. This number rewrites more conveniently as 
\begin{equation} \label{eq:bn}
	b(n)=\frac{n-2}{2}\cdot\left\lfloor \frac{n^2}{4}\right\rfloor
	=
	\begin{cases}
		\tfrac18 n^2(n-2) & \text{ if $n$ is even,} \\
		\tfrac18(n^2-1)(n-2) & \text{ if $n$ is odd.} 
		  \end{cases}
\end{equation}
S\'os conjectured this construction to be optimal, i.e., that 
\begin{equation}\label{eq:Vera}
	\ex(n, \cF)=b(n)
\end{equation}
and that, moreover, $B_n$ is the unique $n$-vertex hypergraph with $b(n)$ edges not 
containing a Fano plane. According to F\"uredi~\cite{Fu}, this conjecture of S\'os was 
widely known since the~1970's. In her problem and survey article~\cite{So76}, which often 
serves as a reference for this problem, she discusses several connections between design 
theory and extremal hypergraph theory, even though \eqref{eq:Vera} does not seem to be 
mentioned there.

The first result in this direction is due to de Caen and F\"uredi~\cite{DeFu00},
who proved that $\pi(\cF)=\tfrac 34$ holds for the Fano plane $\cF$. Their article 
introduced the so-called {\it link multigraph method} on which all further progress
on S\'os's conjecture is based, and which has since then found many further applications 
(see e.g.~\cites{MuRo02, KeMu12}). A few years later it turned out that by combining the work 
in~\cite{DeFu00} with Simonovits' {\it stability method}~\cite{Si68} one can 
prove~\eqref{eq:Vera} for all sufficiently large $n$. 
This was done by F\"uredi and Simonovits in~\cite{FuSi05} and,
independently, by Keevash and Sudakov in~\cite{KeSu05}. It is not straightforward to extract 
optimal quantitative information from either of those articles, but it seems safe to say 
that following~\cite{FuSi05} closely~\eqref{eq:Vera} would be hard to show for all 
$n\ge 10^{100}$ and easy for $n\ge 10^{300}$, while the arguments in~\cite{KeSu05} would probably 
require $n$ to be larger than $10^{900}$. 
      
The main result of the present work proves~\eqref{eq:Vera} for all $n\ge 7$. Furthermore, 
we show that for $n\ge 8$ the balanced, complete, bipartite hypergraph is indeed the 
only extremal configuration. For $n=7$, however, there is a second extremal example, 
which is the hypergraph $J_7$ remaining from the complete hypergraph $K^{(3)}_7$ when 
one deletes all five edges involving a fixed pair of vertices. Plainly $J_7$ has 
$\binom 73-5=30=b(7)$ edges and~$\cF\not\subseteq J_7$ follows from the fact that in 
the Fano plane every pair of points determines a line.      

\begin{thm}\label{thm:main}
	For every integer $n \ge 7$ we have 
	\begin{equation*} \label{eq:Fano}
		\ex(n, \cF)=b(n)=
			\frac{n-2}{2}\cdot \bigg\lfloor\frac{n^2}{4}\bigg\rfloor
				\,,
	\end{equation*}
	where $\cF$ denotes the Fano plane.
	Moreover, for $n\ge 8$ the only extremal hypergraph is the balanced, complete, bipartite 
	hypergraph $B_n$, while for $n=7$ there are exactly two extremal hypergraphs, namely $B_7$ 
	and~$J_7$.
\end{thm}

We would like to point out that this result does not supersede the earlier 
works~\cites{FuSi05, KeSu05}. This is because they also prove the stability result
that every large hypergraph with density~$\tfrac{3}{4}-o(1)$ not containing a Fano plane has 
to look ``almost'' like $B_n$. 

The proof of Theorem~\ref{thm:main} proceeds by induction on $n$  and uses the link 
multigraph method.
Let us mention for completeness that for $n\le 6$ one trivially has $\ex(n, \cF)=\binom n3$,
the unique extremal configuration being the complete hypergraph~$K^{(3)}_n$. 

\subsection*{Organisation.} 
We prove Theorem~\ref{thm:main} in Section~\ref{sec:proof}.
Some auxiliary considerations dealing with small hypergraphs and inductive characterisations 
of balanced, complete, bipartite hypergraphs are gathered in Section~\ref{sec:pre}.
The results on multigraphs we shall require are developed in Section~\ref{sec:multi}.

\section{Preliminaries}\label{sec:pre}

\subsection{Tetrahedra}\label{subsec:Tetra}
The $n$-vertex hypergraphs we need to deal with in the proof of our main result 
will have $b(n)$ edges and, hence, an edge density of $\frac34+o(1)$. 
In view of~\eqref{eq:Raz} such hypergraphs contain tetrahedra provided that $n$ is 
sufficiently large. Later on it will be important to know that this actually holds 
for small values of $n$ as well, which can be seen by means of the following well-known, 
elementary argument.

Starting from the obvious fact $\ex\bigl(4, K^{(3)}_4\bigr)=3$ one uses the monotonicity 
of the sequence 
\[
	n\longmapsto \frac{\ex\bigl(n, K^{(3)}_4\bigr)}{\binom{n}{3}}
\]
in order to obtain 
\[
	\ex\bigl(n, K^{(3)}_4\bigr)\le\frac 3{4}\binom{n}{3}
\]
for every $n\ge 4$. Together with the estimate
\[
	\frac 3{4}\binom{n}{3}
	=
	\frac {n(n-1)(n-2)}8
	<
	\frac {(n+1)(n-1)(n-2)}8
	\overset{\text{\eqref{eq:bn}}}{\le}
	b(n)\,,
\]
which holds for all $n\ge 3$, this leads to the following statement.

\begin{fact} \label{fact:tetra}
	For $n\ge 4$, every hypergraph on $n$ vertices with $b(n)$ edges contains a tetrahedron.
\end{fact}

\subsection{Finding Fano planes} \label{subsec:Fano}

This subsection discusses two ways of looking at the Fano plane $\cF$ that turn out to be helpful
for realising that a given hypergraph $H$ contains a copy of $\cF$. 

The first of them goes back to the work of de Caen and F\"uredi~\cite{DeFu00} and 
reappeared in all subsequent articles addressing the Tur\'an problem for the Fano plane. 
Given a vertex $x$ of an arbitrary hypergraph $H$ one may form its so-called {\it link graph}
with vertex set~$V(H)$ in which two vertices $u$ and $v$ are declared to be adjacent 
if and only if the triple~$uvx$ is an edge of $H$. Now the simple yet important observation
one frequently uses is that if $xyz$ denotes an arbitrary edge of the Fano plane $\cF$,
then the six further edges of~$\cF$ correspond to certain edges of the link graphs
of $x$, $y$, and $z$. Moreover, these edges in the link graphs use four vertices only 
and they form a configuration which has, for obvious reasons, been called {\it ``three 
crossing pairs''} in~\cite{FuSi05} (see Fig.~\ref{fig:Fano-Link}).    

\begin{figure}[ht]
\centering
\begin{subfigure}[b]{0.2\textwidth}
    \centering
	\begin{tikzpicture}
		\coordinate(z) at (0, 0);
		\coordinate(a) at (1.5, 0);
		\coordinate(b) at (3, 0);
		\coordinate(x) at (1.5, {1.5*sqrt(3)});
		\coordinate(y) at (3/4, {sqrt(27)/4});
		\coordinate(c) at (3/2, {sqrt(3)/2});
		\coordinate(d) at (9/4, {sqrt(27)/4});
		
		\path[black,thick]
			(x) edge (z);
		\path[{green!50!black}, thick]
			(x) edge (b)
			(x) edge (a);
		\path[blue,thick]
			(y) edge (b);
		\node [draw,blue, thick, circle through=(a)] at (c) {};
		\path[red, thick]
			(z) edge (b)
			(z) edge (d);

		\fill[{green!50!black}] (x) circle (2pt);
		\fill[blue] (y) circle (2pt);
		\fill[red] (z) circle (2pt);
		\foreach \i in {a, b, c, d}
			\fill  (\i) circle (2pt);			
		\node[red] at (-0.2, 0.2) {$x$};
		\node[blue] at (0.5, 1.3) {$y$};
		\node[{green!50!black}] at (1.3, 2.8) {$z$};
	\end{tikzpicture}
        \caption{Fano plane}
    \end{subfigure}
		\hskip3cm
    \begin{subfigure}[b]{0.2\textwidth}
    \centering\vskip1cm
	\begin{tikzpicture}
	\coordinate(y) at (1, 1);
	\coordinate(w) at (1, 3);
	\coordinate(z) at (3, 1);
	\coordinate(x) at (3, 3);	
		
	\path[red,thick]
		(w) edge (x)
		(y) edge (z)
	;
	\path[blue,thick]
		(w) edge (z)
		(x) edge (y)
	;
	\path[{green!60!black},thick]
		(w) edge (y)
		(x) edge (z)
	;
	\foreach \i in {x, y, z, w}
			\fill  (\i) circle (2pt);
	\end{tikzpicture}
        \caption{Crossing pairs}
    \end{subfigure}
	\caption{The edge $xyz$ and the link graphs of $\color{red}{x}$, $\color{blue}{y}$, 
		and $\color{green!60!black}{z}$.}
	\label{fig:Fano-Link}
\end{figure}	

Another way of locating Fano planes in dense hypergraphs focuses on the link graph
of a single vertex. Plainly, every vertex $x$ of the Fano plane $\cF$ belongs to three
edges of $\cF$, which correspond to a perfect matching $M$ in the link graph of $x$
restricted to the six remaining vertices of $\cF$. There are four further edges in~$\cF$ 
forming a certain tripartite hypergraph~$\cP$, whose partition classes are given by $M$.
Owing to its connection with Pasch's axiom in the axiomatic approach to planar Euclidean 
geometry (see~\cite{Pasch}*{{\normalfont\scshape\centering\S}2, Grundsatz IV}), 
$\cP$ is often called the {\it Pasch hypergraph} (see Fig.~\ref{fig:Pasch}). 

This perspective on the Fano plane is 
especially useful when combined with the stability method, for the Pasch hypergraph
is known to have vanishing Tur\'an density---a fact exploited both in~\cite{FuSi05}
and in~\cite{KeSu05}. In the present work, the Pasch hypergraph plays a much less prominent 
role and it will only be mentioned in the proof of Lemma~\ref{lem:11-18} below.  

\begin{figure}[ht]
\centering
\begin{subfigure}[b]{0.2\textwidth}
    \centering
	\begin{tikzpicture}
		\coordinate(z) at (0, 0);
		\coordinate(a) at (1.5, 0);
		\coordinate(b) at (3, 0);
		\coordinate(x) at (1.5, {1.5*sqrt(3)});
		\coordinate(y) at (3/4, {sqrt(27)/4});
		\coordinate(c) at (3/2, {sqrt(3)/2});
		\coordinate(d) at (9/4, {sqrt(27)/4});

		\path[red, thick]
			(x) edge (b)
			(z) edge (d);
		\path[black,thick]
			(x) edge (a)
			(x) edge (z)
			(y) edge (b);
		\node [draw,red, thick, circle through=(a)] at (c) {};
		\path[black, thick]
			(z) edge (b);

		\fill[{red}] (d) circle (2pt);
		\foreach \i in {a, b, c, x, z, y}
			\fill  (\i) circle (2pt);			
		\node[red] at (2.5, 1.5) {$x$};
	\end{tikzpicture}
    \end{subfigure}
    \hskip2cm
    \begin{subfigure}[b]{0.2\textwidth}
    \centering\vskip1cm
	\begin{tikzpicture}
	\coordinate(z) at (0, 0);
	\coordinate(a) at (2, 0);
	\coordinate(b) at (3.5, 0);
	\coordinate(x) at (2.2, 2.6);
	\coordinate(y) at (1.32, 1.56);	
			
	\draw[name path=xa, thick] (x) -- (a);
	\draw[name path=yb, thick] (y) -- (b);
	\path[name intersections={of=xa and yb, by={c}}];
	\path[black,thick]
		(x) edge (z)
		(x) edge (a)
		(b) edge (y)
		(b) edge (z)
	;
	\path[red,thick]
		(z) edge (c)
		(y) edge (a)
		(b) edge (x);
	\foreach \i in {x, y, z, a, b, c}
			\fill  (\i) circle (2pt);
	\end{tikzpicture}
    \end{subfigure}
	\caption{The Pasch hypergraph contained in the Fano plane.}
	\label{fig:Pasch}
\end{figure}
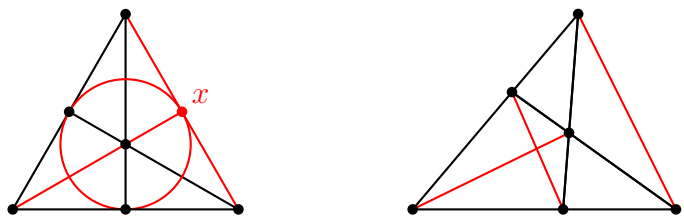
 
\subsection{Small hypergraphs}\label{subsec:small} In this subsection we gather several 
auxiliary statements addressing hypergraphs on $7$ or $8$ vertices.
We begin with the case~$n=7$ of Theorem~\ref{thm:main}, which will later constitute the start
of an induction.   

\begin{lem} \label{lem:n=7}
	Every hypergraph with $7$ vertices and $30$ edges not containing a Fano plane 
	is isomorphic to either $B_7$ or $J_7$.
\end{lem}

\begin{proof}  
	Let $H$ be such a hypergraph with vertex set $[7]$ and write $\overline{H}$
	for its complement, which has $5$ edges.
	
	For every permutation $\pi$ in the symmetric group $S_7$ we denote the number of 
	triples among
	\begin{align*}
		& \pi(1)\pi(2)\pi(3),\,\,\, \pi(3)\pi(4)\pi(5), \,\,\, \pi(1)\pi(5)\pi(6), \,\,\,
		  \pi(1)\pi(4)\pi(7), \\
		& \pi(3)\pi(6)\pi(7), \,\,\, \pi(2)\pi(5)\pi(7), \,\,\,
		\text{ and } \,\,\,
		\pi(2)\pi(4)\pi(6)\,,
	\end{align*}
	which are edges of $\overline{H}$, by $A(\pi)$. As these seven triples form 
	a Fano plane, the number $A(\pi)$ cannot vanish for any $\pi\in S_7$, wherefore
	\[
		\sum_{\pi\in S_7}A(\pi)\ge |S_7|=7!\,.
	\]

	On the other hand, every edge of $\overline{H}$ appears in the above list for 
	precisely $7\cdot 3!\cdot 4!=7!/5$ permutations $\pi$ and a double-counting 
	argument yields
	\[
		\sum_{\pi\in S_7}A(\pi)=\tfrac{7!}5 \cdot e(\overline{H})=7!\,.
	\]

	For these reasons, we have $A(\pi)=1$ for every $\pi\in S_7$.
	If $\overline{H}$ would have two edges intersecting in a single vertex, 
	then an appropriate permutation~$\pi\in S_7$ would satisfy~${A(\pi)\ge 2}$, 
	which has just been proved to be false. Therefore, any two 
	distinct edges of~$\overline{H}$ are either disjoint or they intersect in a pair.

	\begin{figure}[ht]
\centering
\begin{subfigure}[b]{0.3\textwidth}
    \centering
	\begin{tikzpicture}
	\foreach \x in {1,2,...,7}{
		\coordinate (\x) at (90+\x*360/7:1.5);
 	}
	\triple{(1)}{(7)}{(6)}{4.5pt}{1.5pt}{red!70!black}{red!70!black,opacity=0.2};
	\triple{(2)}{(4)}{(3)}{4.5pt}{1.5pt}{red!70!black}{red!70!black,opacity=0.2};
	\triple{(2)}{(5)}{(3)}{4.5pt}{1.5pt}{red!70!black}{red!70!black,opacity=0.2};
	\triple{(2)}{(5)}{(4)}{4.5pt}{1.5pt}{red!70!black}{red!70!black,opacity=0.2};
	\triple{(3)}{(5)}{(4)}{4.5pt}{1.5pt}{red!70!black}{red!70!black,opacity=0.2};
	\foreach \x in {1,2,...,7}{
 		\fill (90+\x*360/7:1.5) circle (2pt);
 	}
	\end{tikzpicture}
        \caption{$\overline{B_7}$}
        \label{fig:7a}
    \end{subfigure}
		\hskip3cm
    \begin{subfigure}[b]{0.3\textwidth}
    \centering\vskip1cm
	\begin{tikzpicture}
 	\foreach \x in {1,2,...,7}{
 			\coordinate (\x) at (90+\x*360/7:1.5);
 	}
	\triple{(4)}{(3)}{(1)}{4.5pt}{1.5pt}{red!70!black}{red!70!black,opacity=0.2};
	\triple{(4)}{(3)}{(2)}{4.5pt}{1.5pt}{red!70!black}{red!70!black,opacity=0.2};
	\triple{(4)}{(3)}{(5)}{4.5pt}{1.5pt}{red!70!black}{red!70!black,opacity=0.2};
	\triple{(4)}{(3)}{(6)}{4.5pt}{1.5pt}{red!70!black}{red!70!black,opacity=0.2};
	\triple{(4)}{(3)}{(7)}{4.5pt}{1.5pt}{red!70!black}{red!70!black,opacity=0.2};
	\foreach \x in {1,2,...,7}{
 		\fill (90+\x*360/7:1.5) circle (2pt);
 	}
	\end{tikzpicture}
        \caption{$\overline{J_7}$}
        \label{fig:7b}
    \end{subfigure}
	\caption{Possibilities for $\overline{H}$.}
	\label{fig:n=7}
\end{figure}
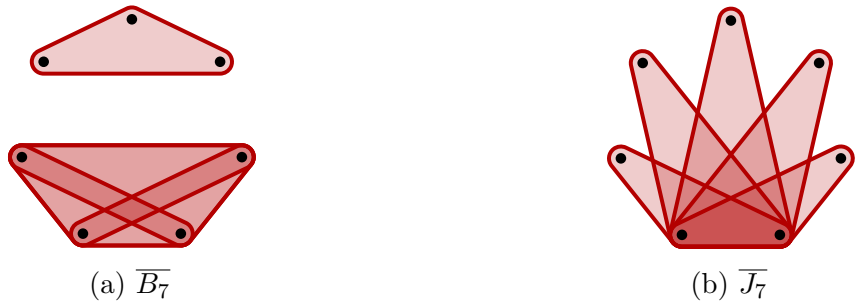	
	A quick case analysis 
	discloses that there are only two hypergraphs on $7$ vertices with $5$ edges having 
	this property,
	namely the disjoint union of a tetrahedron and a single edge 
	(see Fig.~\ref{fig:7a}), 
	and the hypergraph whose edges are the five triples containing a fixed pair of vertices
	(see Fig.~\ref{fig:7b}). 
	In the former case $H$ is isomorphic to $B_7$ and in the latter case one has $H\cong J_7$.
\end{proof}

The next lemma analyses certain Fano-free hypergraphs on $7$ vertices with possibly 
only~$29$ edges. It will allow us later to exclude several 
configurations on six vertices in a hypothetical minimal counterexample to 
Theorem~\ref{thm:main}. In its proof we exploit that every graph on six vertices with eleven 
edges contains a perfect matching. Moreover, the unique graph on six vertices with ten edges 
not containing a perfect matching consist of a $K_5$ plus an isolated vertex. Both facts
can either be proved by a direct case analysis based on Tutte's $1$-factor theorem~\cite{Tutte}
or by plugging $n=6$ and $\beta=2$ into~\cite{B}*{Corollary~II.1.10}   

\begin{lem} \label{lem:11-18}
	Let $H$ be a hypergraph on $7$ vertices not containing a Fano plane. 
	If some vertex $v$ of $H$ satisfies $d(v)\ge 11$ and $e(H\sm v)\ge 18$,
	then $H\sm v$ is isomorphic to $B_6$.  
\end{lem}

\begin{proof}
	Set $K=V(H)\sm \{v\}$ and let $L$ denote the link graph of $v$ restricted to $K$.
	It has~$6$ vertices and at least $11$ edges, and thus it contains a perfect matching $M$, 
	say with edges $x_1x_2$, $x_3x_4$, $x_5x_6$. 
	
	Notice that the complement $H_\star$ of $H\sm v$ has at most two edges. 
	Assuming indirectly that $H\sm v$ is not isomorphic to $B_6$ we know that 
	this complement does not consist of two disjoint edges 
	and thus there is a vertex, say $x_6$, belonging to all edges of $H_\star$. 
	In other words, $\{x_1, \ldots, x_5\}$ is a clique of order $5$ in $H$.
	
	\begin{figure}[ht]
\centering
\begin{tikzpicture}[scale=1]
	\node[ellipse, draw=black, thick,minimum width=8cm,minimum height=4cm, 
		inner sep=0pt][] (a) at (3,2){}; 
	\coordinate(x1) at (1, 3);
	\coordinate(x2) at (1, 1);
	\coordinate(x3) at (3, 3);
	\coordinate(x4) at (3, 1);
	\coordinate(x5) at (5, 3);
	\coordinate(x6) at (5, 1);
	\coordinate(v) at (-2, 2);
	
	\path[{green!60!black},thick]
		(x1) edge (x2)
		(x3) edge (x4)
		(x5) edge (x6)
	;
	\arcThroughThreePoints[red, thick]{x1}{x4}{x5};
	\draw[red, thick] (x2) -- (x6);
	\draw[red, thick] (x1) .. controls (1.6, 3.5) and (2.4,3.5) .. (x3) -- (x6);
	\draw[red, thick] (x5) .. controls (4.6, 3.5) and (3.6,3.5) .. (x3) -- (x2);
	
	\draw[blue, thick] (x1) -- (x5);
	\arcThroughThreePoints[blue, thick]{x6}{x3}{x2};
	
	\draw[blue, thick] (x2) .. controls (1.6,0.5) and (2.4,0.5) .. (x4) -- (x5);
	\draw[blue, thick] (x6) .. controls (4.6,0.5) and (3.6,0.5) .. (x4) -- (x1);	
	\foreach \i in {x1, x2, x3, x4, x5, x6}
			\fill  (\i) circle (2pt);
	\fill[{green!50!black}] (v) circle (2pt);
	\node at (0.7,3.3) {$x_1$};
	\node at (0.7,0.7) {$x_2$};
	\node at (3,0.7) {$x_4$};
	\node at (3,3.3) {$x_3$};
	\node at (5.3,3.3) {$x_5$};
	\node at (5.3,0.7) {$x_6$};
	\node at (-2.2,2) {$\color{green!50!black}{v}$};
	\node at (7,3.1) {$K$};
	\node at (5.5, 2) {$\color{green!50!black}{M}$};
\end{tikzpicture}
\caption{The matching $\color{green!60!black}{M}$ in the link of $\color{green!60!black}{v}$ 
and two Pasch hypergraphs (drawn \textcolor{red}{red} and \textcolor{blue}{blue}).}
\label{fig:11-18}
\end{figure}  

	Now both 
	\begin{align*}
		&x_1x_3x_5, \,\,\, x_1x_4x_6, \,\,\, x_2x_3x_6, \,\,\, x_2x_4x_5 
		\intertext{and} 
		&x_2x_4x_6, \,\,\, x_2x_3x_5, \,\,\, x_1x_4x_5, \,\,\, x_1x_3x_6
	\end{align*} 
	are edge configurations forming Pasch hypergraphs that together with the 
	matching $M$ in the link of $v$ would yield a Fano plane (see Figure~\ref{fig:11-18}). 
	Thus both of the above disjoint rows contain a triple which fails to be an edge of $H$. 
	On the other hand the complement~$H_\star$ has already been observed to possess have 
	at most two edges.
	
	So without loss of generality we may suppose that the edges of $H_\star$ are 
	$x_1x_4x_6$ and $x_2x_4x_6$. Now $x_1$ and $x_2$ are the only vertices 
	of $H_\star$ having degree $1$. If there were a different perfect matching~$M'$ in $L$ 
	not pairing these two vertices with each other, we could repeat the entire 
	argument with $M'$ in place of $M$ and would thus find a Fano plane in $H$.
	
	This shows that all perfect matchings of $L$ use the edge $x_1x_2$. Hence the graph
	$L\sm x_1x_2$ with~$6$ vertices and at least $10$ edges has no perfect matchings 
	at all, which is only possible if this graph consists of a $K_5$ and an isolated 
	vertex. As the edge $x_1x_2$ cannot belong to this $K_5$, the isolated vertex must 
	be either $x_1$ or $x_2$. In both cases 
	\[
		x_1x_2, \,\,\, x_3x_5, \,\,\, x_4x_6
	\]
	is a perfect matching in $L$. Together with the Pasch hypergraph
	\[
		x_1x_3x_4, \,\,\, x_1x_5x_6, \,\,\, x_2x_3x_6, \,\,\, x_2x_4x_5
	\]
	it leads to a Fano plane in $H$, which contradicts the hypothesis. 
	Thus we have indeed $(H\sm v)\cong B_6$.   	  
\end{proof}

Finally, the last statement of this subsection will allow us later to eliminate a 
somewhat annoying case that arises in the induction step from $7$ to $8$ due to 
the non-uniqueness of the extremal hypergraph on $7$ vertices.
 
\begin{fact} \label{fact:8-46}
	Let $H$ be a hypergraph on $8$ vertices not containing a Fano plane. 
	If $K_6^{(3)}\subseteq H$, then $e(H)\le 46<b(8)$.
\end{fact}

\begin{proof}
	Write $V(H)=K\cup\{x, y\}$, where $K$ induces a $K_6^{(3)}$ in $H$. 
	By Lemma~\ref{lem:11-18} applied to~$H\sm y$ there are at most $10$ 
	edges containing $x$ but not $y$. Similarly, there are at most~$10$ 
	edges containing $y$ but not $x$. Finally, $H$ can have at most $|K|=6$
	edges containing both~$x$ and~$y$. So altogether we have indeed
	\[
		e(H)\le 20+10+10+6=46<48=b(8)\,. \qedhere
	\]
\end{proof}

\subsection{Characterisations of \texorpdfstring{$B_n$}{\it B(n)}}   
  
In our inductive proof of Theorem~\ref{thm:main} we will consider a hypergraph
$H$ on some number $n\ge 8$ of vertices with $b(n)$ edges and $\cF\not\subseteq H$. 
These assumptions will be shown to entail some strong structural properties of $H$ 
and the purpose of this subsection is to check that we can actually conclude $H\cong B_n$
from those properties. 

This is much easier when the number of vertices is even. 

\begin{lem} \label{lem:even} 
	Suppose that $n\ge 6$ is even and that $H$ is a hypergraph on $n$ vertices. 
	If for every vertex $v$ of $H$ the hypergraph $H\sm v$ is isomorphic 
	to $B_{n-1}$, then $H\cong B_n$.
\end{lem}

\begin{proof}
	Let $y\in V(H)$ be arbitrary. Since $H\sm y$ is isomorphic to $B_{n-1}$, there 
	exists a partition $V(H)\sm\{y\}=X\dcup Y$ with $|X|=\frac n2$ and $|Y|=\frac n2-1$
	such that $X$ and $Y$ are independent sets in $H$. The same argument applies to every 
	$y'\in Y$. Since $B_{n-1}$ has a unique independent set of size $\frac n2$, the outcome 
	must be the partition 
	\[
		V(H)\sm\{y'\}=X\dcup \bigl(Y\cup\{y\}\sm \{y'\}\bigr)
		\quad \text{for each } 
		y'\in Y\,.
	\]
	This proves that $H$ is isomorphic to $B_n$ with vertex classes 
	$X$ and $Y\cup \{y\}$. 
\end{proof}

To handle the case where the number of vertices is odd we shall require the following lemma.
Its initial assumption concerning the case $n=8$ will turn out to be harmless, as
we will already know its truth when using the lemma for the first time. 

\begin{lem} \label{lem:odd}
	Assume that Theorem~\ref{thm:main} holds for $n=8$.
	Now let $n\ge 9$ be odd and let $H$ be a hypergraph on $n$ vertices with $b(n)$ edges
	which does not contain a Fano plane.
	Suppose that whenever a four-element set $K\subseteq V(H)$ induces a tetrahedron in $H$
	\begin{enumerate}[label=\rmlabel]
		\item\label{it:oddi} we have $(H\sm K)\cong B_{n-4}$
		\item\label{it:oddii} and every $v\in V(H\sm K)$ has degree exactly 5 in $K$. 
	\end{enumerate}
	Then $H$ is isomorphic to $B_n$.
\end{lem}
  
\begin{proof}
	Recall that by Fact~\ref{fact:tetra} there is a tetrahedron contained in $H$,
	say with vertex set $K\subseteq V(H)$. Owing to condition~\ref{it:oddi} there 
	is a partition $V\sm K=X\dcup Y$ witnessing that~$H\sm K$ is indeed isomorphic 
	to $B_{n-4}$. Notice that due to $n\ge 9$ we may suppose that~$|X|\ge 2$ and 
	$|Y|\ge 3$. 
	
	Now consider any four distinct vertices $x, x'\in X$ and $y, y'\in Y$. By 
	clause~\ref{it:oddii} applied to the tetrahedra~$K$ and $K'=\{x, x', y, y'\}$ we obtain
	\[
		e(K\cup K')=e(K)+e(K')+5(|K|+|K'|)=2\cdot 4+5\cdot 8=48=b(8)
	\]
	and by the hypothesised validity of Theorem~\ref{thm:main} for hypergraphs on $8$ 
	vertices it follows that~$K\cup K'$ induces a copy of $B_8$ in $H$. As $K$ induces
	a tetrahedron, there exists an enumeration $K=\{v_1, v_2, v_3, v_4\}$ such that 
	the two independent $4$-sets of this $B_8$ are, possibly after relabelling $y$ and $y'$,  
	\begin{enumerate}[label=\alabel]
		\item\label{it:4aa} either $\{v_1, v_2, x, x'\}$ and $\{v_3, v_4, y, y'\}$
		\item\label{it:4bb} or $\{v_1, v_2, x, y\}$ and $\{v_3, v_4, x', y'\}$.
	\end{enumerate}
\vskip-4em
	\begin{figure}[ht]
\centering
\begin{tikzpicture}
	\coordinate(v1) at (1, 3);
	\coordinate(v2) at (3, 3);
	\coordinate(v3) at (1, 1);
	\coordinate(v4) at (3, 1);
	\coordinate(x) at (5, 3);
	\coordinate(x') at (5, 1);
	\coordinate(y) at (7, 4);
	\coordinate (y') at (7, 2);
	\coordinate (y'') at (7, 0);
	
	\quadrupel{(v1)}{(v2)}{(v4)}{(v3)}{4.5pt}{2.25pt}{yellow!60!black}{yellow ,opacity=0.5};
	\quadrupel{(x)}{(y)}{(y')}{(x')}{4.5pt}{2.25pt}{yellow!60!black}{yellow ,opacity=0.5};
	
	\draw[red!70!black, line width=2pt, rounded corners, line cap=round]
		([yshift=1.5pt]v3) -- ([yshift=1.5pt]x') -- (y');
	\draw[red!70!black, line width=2pt, dashed, rounded corners, line cap=round]
		([yshift=-1.5pt]v3) -- ([yshift=-1.5pt]x') -- (y'');
	\draw[red!70!black, line width=2pt, rounded corners, line cap=round]
		([yshift=1.5pt]v1) -- ([yshift=1.5pt]x) -- (y);
	
	\node at (0.7,3.3) {$v_1$};
	\node at (3.3,3.3) {$v_2$};
	\node at (0.7,0.7) {$v_3$};
	\node at (3.3,0.7) {$v_4$};
	\node at (4.6,3.3) {$x$};
	\node at (4.6,0.7) {$x'$};
	\node at (7.5,4) {$y$};
	\node at (7.6,0) {$y''$};
	\node at (7.6,2) {$y'$};
	
	\node at (5,4) {$X$};
	\node at (7,5) {$Y$};
	\node at (2,2) {$K$};
	\node at (6,2.5) {$K'$};
	
	\node[ellipse, draw=black, thick, minimum width=.6cm,minimum height=3cm, 
		inner sep=0pt][] (a) at (5, 2){}; 
	\node[ellipse,  draw=black, thick, minimum width=.8cm,
		minimum height=5cm, inner sep=0pt][] (a) at (7, 2) {};
	\foreach \i in {v1, v2, v3, v4, x, x', y'', y, y'}
			\fill  (\i) circle (2pt);
\end{tikzpicture}
\caption{The impossible case~\ref{it:4bb}. Tetrahedra are drawn 
as yellow quadruples and independent sets as red lines.}
\label{fig:case-b}
\end{figure}
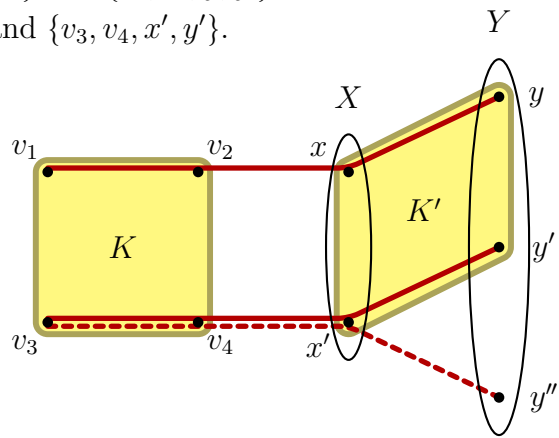

	Now assume for the sake of contradiction that the latter possibility occurs
	(see Fig.~\ref{fig:case-b}). 
	Let $y''\in Y$ be an arbitrary vertex distinct from $y$ and $y'$. When applying
	the argument of the foregoing paragraph to $\{x, x', y, y''\}$ instead of $K'$
	we still have the independent set~$\{v_1, v_2, x, y\}$ and, consequently,  
	$\{v_3, v_4, x', y''\}$ is independent as well. But now, as the edges $v_3x'y'$ 
	and $v_3x'y''$ are missing, the degree of $v_3$ in the tetrahedron $\{x, x', y', y''\}$ 
	is at most $4$, which violates condition~\ref{it:oddii}.  
	This proves that alternative~\ref{it:4bb} is indeed impossible.
	
	Summarising the discussion so far, we know that depending on any four distinct vertices 
	$x, x'\in X$ and $y, y'\in Y$ there is an enumeration $K=\{v_1, v_2, v_3, v_4\}$
	such that the two independent $4$-sets of the copy of $B_8$ induced by 
	$K\cup\{x, x', y, y'\}$ are as mentioned in~\ref{it:4aa}.
	
	Now if we keep $y$ and $y'$ fixed and let the pair $x$, $x'$ vary through $X$
	we will always get the same independent set $\{v_3, v_4, y, y'\}$ and thus 
	the entire set $X\cup\{v_1, v_2\}$ is independent in $H$. Similarly, 
	$Y\cup\{v_3, v_4\}$ is independent as well. Consequently $H$ is indeed isomorphic to 
	$B_n$ with partition classes $X\cup\{v_1, v_2\}$ and  $Y\cup\{v_3, v_4\}$. 
\end{proof}  
    
\section{Multigraphs} \label{sec:multi}

This section builds upon~\cite{FuSi05}*{Section 2-4} and collects several extremal results 
on multigraphs that will be applied at a later occasion to certain link multigraphs arising in 
hypergraphs not containing Fano planes. 

\begin{dfn}
	For a positive integer $p$, a $p$-tuple $\seq{G}=(G_1, \ldots, G_p)$ of graphs on the
	same vertex set $V(\seq{G})$ will be referred to as a {\it $p$-multigraph}. 
\end{dfn}

Extending some pieces of graph theoretic notation to the context of multigraphs, we will write 
$e(\seq{G})=\sum_{i=1}^p e(G_i)$ for the {\it total number of edges} of a $p$-multigraph
$\seq{G}=(G_1, \ldots, G_p)$. Similarly, for every $X\subseteq V(\seq{G})$ we put 
$e(X)=\sum_{i=1}^p e_{G_i}(X)$ and if the members of $X$ are enumerated explicitly
we will omit a pair of curly braces and write, e.g., $e(x, y, z)$ instead of the more baroque 
$e(\{x, y, z\})$. In the special case of two-element sets, the number $e(x, y)$ will be 
called the {\it multiplicity} of the pair $xy$.

With each $p$-multigraph one can associate a corresponding {\it weighted graph} $(V, e)$
given by the set of vertices $V=V(\seq{G})$ and the multiplicity function 
$(x, y)\longmapsto e(x, y)$.
There is a rich literature on extremal problems in weighted graphs and the topic is 
studied both for its own sake (see e.g. \cites{RoSi95, FK02}) and due to its applicability
to other parts of extremal combinatorics, such as Tur\'an's hypergraph problem 
and the Ramsey-Tur\'an theory of graphs (see e.g.~\cites{EHSS, LR-a, LR-b}). 

The main difference between multigraphs and weighted graphs is that the  
former do also keep track of the sets $M(x, y)\subseteq [p]$ containing for every pair 
$xy$ of vertices those indices $i\in [p]$ for which $xy$ is an edge of $G_i$. 
Therefore, there is a richer variety of extremal questions that can be asked in 
the setting of multigraphs. The following such problem is closely tied to the 
Tur\'an number of the Fano plane.

\begin{dfn}
	For $p\ge 3$ a $p$-multigraph $\seq{G}=(G_1, \ldots, G_p)$ is said to contain 
	{\it three crossing pairs} (see Figure~\ref{fig:crossing}) if there are three 
	distinct indices $i, j, k\in [p]$ and four distinct vertices $w, x, y, z\in V(\seq{G})$ 
	such that 
	\begin{enumerate}
		\item[$\bullet$] $wx, yz\in E(G_i)$;
		\item[$\bullet$] $wy, xz\in E(G_j)$;
		\item[$\bullet$] and $wz, xy\in E(G_k)$.
	\end{enumerate}
	The maximum total number of edges that a $p$-multigraph on $n$ vertices can have 
	without containing three crossing pairs is denoted by $f_p(n)$.
\end{dfn}
\vskip-.8em
\begin{figure}[ht]
\centering
\begin{tikzpicture}[scale=0.6]
	\coordinate(y) at (0, 0);
	\coordinate(w) at (0, 3);
	\coordinate(z) at (3, 0);
	\coordinate(x) at (3, 3);	
		
	\path[red,thick]
		(w) edge (x)
		(y) edge (z)
	;
	\path[blue,thick]
		(w) edge (y)
		(x) edge (z)
	;
	\path[{green!60!black},thick]
		(w) edge (z)
		(x) edge (y)
	;
	\foreach \i in {x, y, z, w}
			\fill  (\i) circle (3pt);
		
	\node at (-0.4, -0.4) {$y$};
	\node at (-0.4, 3.2) {$w$};
	\node at (3.4, -0.4) {$z$};
	\node at (3.4, 3.2) {$x$};
\end{tikzpicture}
\caption{Three crossing pairs in $({\color{red}{G_i}}, 
	{\color{blue}{G_j}}, {\color{green!60!black}{G_k}})$.}
\label{fig:crossing}
\end{figure}   
The function $f_4(\cdot)$ was determined by F\"uredi and Simonovits 
in~\cite{FuSi05}*{Theorem 2.2}. Their result plays an important role 
in the proof of our main result and reads as follows.

\begin{thm}\label{thm:f4}
	For every $n\ge 4$ one has 
	\[
		f_4(n) = 2\binom{n}{2} + 2\bigg\lfloor \frac{n^2}{4} \bigg\rfloor\,.
	\]
\end{thm}

We would like to mention that F\"uredi and Simonovits also obtained a characterisation 
of the extremal configurations on~$n\ge 8$ vertices (see Figure~\ref{fig:FuSi}). 
Namely, if ${\seq{G}=(G_1, G_2, G_3, G_4)}$ denotes a $4$-multigraph 
on at least~$8$ vertices with~$f_4(n)$ edges that does not contain three crossing pairs, then 
there are a partition $V(\seq{G})= X\dcup Y$ and a permutation $\pi$ in the symmetric 
group $S_4$ such that 
\begin{enumerate}
	\item[$\bullet$] $|X|=\big\lfloor \tfrac{n}{2} \big\rfloor$, 
		$|Y|=\big\lfloor \tfrac{n+1}{2} \big\rfloor$,
	\item[$\bullet$] $E\bigl(G_{\pi(1)}\bigr)=E\bigl(G_{\pi(2)}\bigr)=X^{(2)}\cup K(X, Y)$, 	
	\item[$\bullet$] and $E\bigl(G_{\pi(3)}\bigr)=E\bigl(G_{\pi(4)}\bigr)
			=Y^{(2)}\cup K(X, Y)$,
\end{enumerate}
where $K(X, Y)$ denotes the collection of all pairs $xy$ with $x\in X$ and $y\in Y$.

\begin{figure}[ht]
\centering
\begin{tikzpicture}
	\node at (4,1) {$|X|=\big\lfloor \tfrac{n}{2} \big\rfloor$};
	\node at (4,-1) {$|Y|=\big\lfloor \tfrac{n+1}{2} \big\rfloor$};
	\node at (7, 0) {${\color{brown}{K(X, Y)}}
		={\color{red}{E(G_1)}}\cap{\color{red}{E(G_2)}}\cap
		{\color{green!50!black}{E(G_3)}}\cap{\color{green!50!black}{E(G_4)}}$};
	\node (rect) at (0,0) [fill={brown!50!white}, draw=brown,thick,minimum width=4cm,
		minimum height=1.6cm] {};
	\node[ellipse,fill={red!50!white}, draw=red, thick,minimum width=4cm,minimum height=0.7cm, 
		inner sep=0pt][] (a) at (90:0.8){}; 
	\node[ellipse,fill={green!75!white} ,draw=green, thick,minimum width=4cm,
		minimum height=0.7cm, inner sep=0pt][] (a) at (270:0.8) {};
\end{tikzpicture}
\caption{An extremal $4$-multigraph $({\color{red}{G_1}}, 
	{\color{red}{G_2}}, {\color{green!60!black}{G_3}}, {\color{green!60!black}{G_4}})$
	with $\pi=\mathrm{id}$, 
	${{\color{red}{G_1}}={\color{red}{G_2}}}$, 
	and~ 
	${\color{green!60!black}{G_3}}={\color{green!60!black}{G_4}}$.}
\label{fig:FuSi}
\end{figure}

It can be shown that this characterisation of the extremal configurations extends to
the case~$n=7$ as well, but for $n\in \{4, 5, 6\}$ further extremal multigraphs are 
mentioned in~\cite{FuSi05}.

\smallskip

\centerline{$*\,\,\,\,\,*\,\,\,\,\,*$}

\smallskip

\goodbreak
The remainder of this section deals with the function $f_5(\cdot)$. Two instructive examples 
of $5$-multigraphs without three crossing pairs are the following. 
\begin{enumerate}
	\item[$\bullet$] Let $G_1=G_2=G_3=G_4=G_5$ be a $K_4$-free Tur\'an graph on $n$ vertices.
		Notice that this $5$-multigraph has $5\big\lfloor\frac{n^2}{3}\big\rfloor$ edges.
	\item[$\bullet$] Let $\seq{G_*}=(G_1, G_2, G_3, G_4)$ be an extremal $4$-multigraph
		without three crossing pairs with vertex partition $V(\seq{G_*})=X\dcup Y$ as 
		described earlier, take $G_5$ to be the complete bipartite graph
		between $X$ and $Y$ and consider $\seq{G}=(G_1, \ldots, G_5)$. Clearly the 
		$5$-multigraph $\seq{G}$ does not contain three crossing pairs either and its 
		number of edges is $2\binom{n}{2} + 3\big\lfloor \frac{n^2}{4} \big\rfloor$.
\end{enumerate}

These examples demonstrate 
\begin{equation}\label{eq:f5}
	f_5(n)\ge \max\left( 5 \bigg\lfloor\frac{n^2}{3}\bigg\rfloor, 
		2\binom{n}{2} + 3\bigg\lfloor \frac{n^2}{4} \bigg\rfloor  \right) 
\end{equation}
and, as a matter of fact, we can show that equality holds for every $n$. 
Our proof of this statement is, however, quite laborious 
and relies on extensive case distinctions. 
For this reason we will state and prove below a weaker result on $f_5(\cdot)$ 
which still suffices for the application we have in mind.

\begin{prop} \label{prop:weak-f5}
	We have $f_5(n)\le \tfrac 14(7n^2-n)$ for every natural number $n\ge 3$.
\end{prop}

Before we turn to the proof of this fact we take a closer look at the case $n=4$.

\begin{lem} \label{lem:4vertices}
	Let $\seq{G}=(G_1, \ldots, G_5)$ be a $5$-multigraph on four vertices not containing 
	three crossing pairs and set $e=e(\seq{G})$.
	\begin{enumerate}[label=\rmlabel]
	\item\label{it:4a} If $e\ge 23$, then there exists an enumeration $V(\seq{G})=\{w, x, y, z\}$
		such that
		\[
			e(w, x)+e(y, z)\le 5\,.
		\]
	\item\label{it:4b} If $e\ge 22$, then there exist two distinct vertices $u$ and $v$ with 
		$e(u, v)=5$. 
	\end{enumerate} 
\end{lem}

\begin{proof}
	Write $V(\seq{G})=\{w, x, y, z\}$ and define $a=e(w, x)+e(y, z)$, $b=e(w, y)+e(x, z)$,
	as well as $c=e(w, z)+e(x, y)$ to be the sums of the multiplicities of the three pairs 
	of disjoint edges. By symmetry we may suppose that the enumeration of $V(\seq{G})$
	we started with has been chosen in such a way that $a\le b\le c$ holds. 

	Now suppose for the sake of contradiction that 
	\begin{equation*} \tag{$\star$}
		a\ge 6, \quad b\ge 7, \quad \text { and } \quad c\ge 8\,.
	\end{equation*}  
	Due to $a\ge 6$ there is an index $i\in[5]$ such that $wx$ and $yz$ are edges of $G_i$.
	Similarly,~$b\ge 7$ implies that there are at least two indices $j\in [5]$ with the property 
	that $wy$ and $xz$ are edges of $G_j$ and, hence, at least one of them is distinct from $i$.
	Proceeding in the same way with $c\ge 8$ one finds an index $k\ne i, j$ 
	for which $wz$ and $xy$ are edges of $G_k$. 
	We have thereby found three crossing pairs in $(G_i, G_j, G_k)$
	and this contradiction proves that $(\star)$ is indeed false.
	 
	Now part~\ref{it:4a} of the lemma follows from the observation that $a+b+c=e\ge 23$
	and $10\ge c\ge b\ge a$ entail $c\ge 8$ and $b\ge 7$. So the failure of $(\star)$
	yields $a\le 5$, as desired. 

	For the proof of part~\ref{it:4b} we notice that $a+b+c=e\ge 22$
	and $c\ge b\ge a$ still imply~$c\ge 8$. The falsity of $(\star)$
	shows that at least one of the estimates $a\le 5$ or $b\le 6$ holds.
	In both cases we obtain $c\ge 9$, meaning that at least one of the two pairs 
	$wz$ or $xy$ has multiplicity~$5$.
\end{proof}

\begin{proof}[Proof of Proposition~\ref{prop:weak-f5}]
	The trivial bound $f_5(3)\le 5\binom{3}{2}=15$ shows that our claim holds for $n=3$.
	Next, an easy averaging argument yields $f_5(n)\le \frac54 f_4(n)$ for every natural 
	number $n$. Due to Theorem~\ref{thm:f4} and~\eqref{eq:f5} this gives the exact values 
	\begin{equation}\label{eq:456}
		f_5(4)=25 \quad \text{ and } \quad f_5(5)=40\,,	
	\end{equation}
	which establish the desired estimate for $n\in\{4, 5\}$.  
	Arguing indirectly we now let $n\ge 6$ denote the least integer for which there exists 
	a $5$-multigraph $\seq{G}=(G_1, \ldots, G_5)$ on $n$ vertices with more than 
	$\frac 14(7n^2-n)$ edges that does not contain three crossing pairs. 
	For every set $X\subseteq V=V(\seq{G})$ of vertices we 
	shall write $e^+(X)=e(\seq{G})-e(V\sm X)$ for the total number of edges having at 
	least one endvertex in~$X$. As long as $0<|X|\le n-3$ the minimality of $n$ yields 
	\[ 
		e(V\sm X)\le \tfrac14\bigl(7(n-|X|)^2-(n-|X|)\bigr)\,,
	\]
	whence 
	\[
		e^+(X)> \tfrac14 (14n|X|-7|X|^2-|X|)\,. 
	\]
	In particular, we obtain 
	\begin{equation}\label{eq:e+}
		e^+(X)\ge \begin{cases}
		\frac 12(7n-3) & \text{ if } |X|=1, \cr
		7n- 7 & \text{ if } |X|=2, \cr
		\frac{21}2 n-16 & \text{ if } |X|=3. 
		\end{cases}
	\end{equation}

	Owing to $e(\seq{G})>\frac 72\binom{n}{2}$ the average edge multiplicity in $\seq{G}$
	is greater than $\frac 72$. Therefore, there exist a set $Q\subseteq V$ consisting  
	of four vertices with $e(Q)> 6\cdot \frac 72 =21$, and by 
	Lemma~\ref{lem:4vertices}\ref{it:4b} it follows that there are two distinct 
	vertices $x$ and $y$ with $e(x, y)=5$. 
	
	According to~\eqref{eq:e+} we have 
	\[
		\sum_{z\in V\sm\{x, y\}} \bigl(e(x, z)+e(y, z)\bigr)= e^+(x, y)-5 \ge 7n-12 > 7(n-2)\,.
	\]
	Consequently, there exists a vertex $z$ distinct from $x$ and $y$ with 
	$e(x, z) + e(y, z) \ge 8$. 
	
	Altogether we have thereby shown that there exist triples $(x_*, y_*, z_*)$ 
	of distinct vertices with 
	\[
		e(x_*, y_*)=5 
		\quad \text { and } \quad
		e(x_*, z_*) + e(y_*, z_*) \ge 8
	\]
	and for the remainder of the proof we fix one such triple with the additional 
	property that $e(x_*, z_*) + e(y_*, z_*)\ge 8$ is maximal. Set $\alpha= e(x_*, z_*)$
	as well as $\beta=e(y_*, z_*)$ and observe that we may suppose $\alpha\ge \beta$
	for reasons of symmetry. Clearly $(\alpha, \beta)$ is one of the four ordered
	pairs $(5, 5)$, $(5, 4)$, $(5, 3)$, or $(4, 4)$. 
	
	Because of 
	\begin{align} \label{eq:exyz}
		 \sum_{v\in V\sm\{x_*, y_*, z_*\}} \bigl(e(v, x_*)+e(v, y_*)+e(v, z_*)\bigr)
		&= e^+(x_*, y_*, z_*) - (5+\alpha+\beta) \\ \notag
		&\overset{\text{\eqref{eq:e+}}}{\ge}
			 \bigl(\tfrac{21}{2} n-16\bigr)-15 
			 > 
			 10 (n-3)
	\end{align}
	there exists a vertex $v_*\ne x_*, y_*, z_*$ satisfying 
	\begin{equation} \label{eq:v11}
	e(v_*, x_*)+e(v_*, y_*)+e(v_*, z_*)\ge 11\,.
	\end{equation}

	By applying the left part of~\eqref{eq:456} to the quadruple 
	$\{v_*, x_*, y_*, z_*\}$ we learn 
	\[
		\alpha+\beta\le 25-11-5=9\,, 
	\]
	meaning that the pair $(\alpha, \beta)$ cannot be $(5, 5)$. 
	
	\begin{figure}[ht]
\centering
\begin{tikzpicture}[scale=0.8]
	\coordinate(x) at (0, 0);
	\coordinate(y) at (0, 3);
	\coordinate(z) at (3, 0);
	\coordinate(v) at (3, 3);	
		
	\draw[line width=1pt] (x) -- (y);
	\draw[line width=1pt] (x) -- (z);
	\draw[line width=1pt] (y) -- (z);
	\draw[dashed, line width=1pt] (x) -- (v);
	\draw[dashed, line width=1pt] (y) -- (v);
	\draw[dashed, line width=1pt] (z) -- (v);
	
	\foreach \i in {x, y, z, v}
			\fill  (\i) circle (2.5pt);
		
	\node at (-0.4, -0.4) {$y_*$};
	\node at (-0.4, 3.2) {$x_*$};
	\node at (3.4, -0.4) {$z_*$};
	\node at (3.4, 3.2) {$v_*$};
	
	\node at (-0.4, 1.6) {$5$};
	\node at (1.6, -0.3) {$4$};
	\node at (1.2, 2.2) {$4$};
\end{tikzpicture}
\caption{The case $\alpha=\beta=4$.}
\label{fig:v11}
\end{figure}
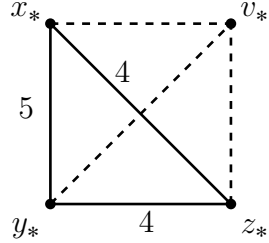
	Assume next that $\alpha=\beta=4$ (see Figure~\ref{fig:v11}), which yields 
	$e(v_*, x_*, y_*, z_*)\ge 13+11=24$.
	Due to Lemma~\ref{lem:4vertices}\ref{it:4a} it follows that either 
	$e(v_*, x_*)\le 1$, $e(v_*, y_*)\le 1$, or $e(v_*, z_*)=0$. 
	The last alternative would contradict~\eqref{eq:v11}, so by symmetry we may suppose 
	that $e(v_*, y_*)\le 1$. Invoking~\eqref{eq:v11} once more we infer that  
	$e(v_*, x_*)=e(v_*, z_*)=5$. 
	But now the edges of the triangle $(v_*, x_*, z_*)$ have multiplicities $5$, $5$, and $4$,
	contrary to the maximal choice of $\alpha+\beta$. 
	We have thereby proved that $(\alpha, \beta)\ne (4, 4)$.    
	
	For these reasons, it must be the case $\alpha=5$ and $\beta\in \{3, 4\}$ (see 
	Figure~\ref{fig:w15}).
	Adding 
	\[
		 \sum_{v\in V\sm\{x_*, y_*, z_*\}} e(v, x_*) 
		 = e^+(x_*) - 10 
		 \overset{\text{\eqref{eq:e+}}}{\ge}
		 \tfrac 72 n - \tfrac{23}2
	\]
	to~\eqref{eq:exyz} we infer
	\begin{align*}
		\sum_{v\in V\sm\{x_*, y_*, z_*\}} \bigl(2e(v, x_*)+e(v, y_*)+e(v, z_*)\bigr)
		& \ge \bigl(\tfrac{21}{2} n-16-14\bigr)+\bigl(\tfrac72 n- \tfrac{23}2\bigr) \\
		&>14(n-3)\,,
	\end{align*}
	which shows that there exists a vertex $w_*\ne x_*, y_*, z_*$ such that
	\begin{equation}\label{eq:w}
		2e(w_*, x_*)+e(w_*, y_*)+e(w_*, z_*)\ge 15\,.
	\end{equation}

	In particular, we have $e(w_*, x_*)+e(w_*, y_*)+e(w_*, z_*)\ge 10$ and, consequently,  
	\[
		e(w_*, x_*, y_*, z_*)\ge 13+10=23\,.
	\]

	\begin{figure}[ht]
\centering
\begin{tikzpicture}[scale=0.8]
	\coordinate(x) at (0, 0);
	\coordinate(y) at (0, 3);
	\coordinate(z) at (3, 0);
	\coordinate(v) at (3, 3);	
		
	\draw[line width=1pt] (x) -- (y);
	\draw[line width=1pt] (x) -- (z);
	\draw[line width=1pt] (y) -- (z);
	\draw[dashed, line width=1pt] (x) -- (v);
	\draw[dashed, line width=1pt] (y) -- (v);
	\draw[dashed, line width=1pt] (z) -- (v);
	
	\foreach \i in {x, y, z, v}
			\fill  (\i) circle (2.5pt);
		
	\node at (-0.4, -0.4) {$y_*$};
	\node at (-0.4, 3.2) {$x_*$};
	\node at (3.4, -0.4) {$z_*$};
	\node at (3.4, 3.2) {$w_*$};
	
	\node at (-0.3, 1.6) {$5$};
	\node at (1.6, -0.4) {$\beta\in\{3, 4\}$};
	\node at (1.2, 2.2) {$5$};
\end{tikzpicture}
\caption{The case $\alpha=5$ and $\beta<5$.}
\label{fig:w15}
\end{figure}
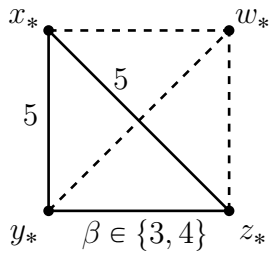

	Appealing to Lemma~\ref{lem:4vertices}\ref{it:4a} again we deduce that 
	at least one of the three cases 
	\[
		e(w_*, x_*)\le 2, \quad e(w_*, y_*)=0, \quad \text{ or } \quad e(w_*, z_*)=0
	\]
	occurs. The first of them is incompatible with~\eqref{eq:w},
	so by symmetry we may suppose that $e(w_*, y_*)=0$. In combination with~\eqref{eq:w}
	this yields~$e(w_*, x_*)=e(w_*, z_*)=5$. But now the triangle $(w_*, x_*, z_*)$
	contradicts the supposed maximality of $\alpha+\beta$. 
\end{proof}

Let us finally summarise the properties of $f_5$ that shall be utilised in the next section. 
\begin{cor}\label{conc:bf}
	If $n\ge 9$ is odd, then 
	\begin{enumerate}[label=\alabel]
		\item\label{it:a} $b(n-5)+f_5(n-5)+7(n-5)+10 < b(n)$ and 
		\item\label{it:b} $\bigl(b(n-6)+ \frac12(n-9)\bigr)+f_5(n-6) +\binom{n-6}{2}
			+10(n-6) +20 <b(n)$.
	\end{enumerate}
\end{cor}

\begin{proof}
	Due to the explicit formula~\eqref{eq:bn} for $b(n)$ and the estimate on $f_5(n-5)$ 
	provided by Proposition~\ref{prop:weak-f5}, part~\ref{it:a} is a consequence of
	$1<\frac 18(n-5)^2$, which is trivially valid. Similarly, 
	part~\ref{it:b} reduces to $0< \frac14(3n-23)$,
	which is likewise obvious. 
\end{proof}

\section{Proof of the Main Theorem} \label{sec:proof}
	This entire section is dedicated to the proof of Theorem~\ref{thm:main}, which 
	proceeds by induction on $n$.
	Since the base case $n=7$ was already treated in 
	Lemma~\ref{lem:n=7}, we may suppose that~$n\ge 8$ and that~\eqref{eq:Fano}
	as well as our statement addressing the extremal hypergraphs hold for every $n'\in [7, n)$
	in place of $n$. Now let $H=(V, E)$ be a hypergraph on $|V|=n$ vertices with $|E|=b(n)$
	edges that does not contain a Fano plane. We are to prove that~$H\cong B_n$. 
	Let us distinguish two cases according to the parity of $n$. 
	
	\smallskip

	{\it \hskip1em First Case: $n\ge 8$ is even.} 

	\medskip	
	  
	For every vertex $v\in V$ we have $e(H\sm v)\le b(n-1)$, since otherwise the
	induction hypothesis would yield a Fano plane in $H\sm v$. 
	This shows that 
	\[
		d(v)\ge b(n)-b(n-1)=3\binom{n/2}2=\frac {3|E|}n
	\]
	holds for every $v\in V$. Due to $\sum_{v\in V}d(v)= 3|E|$ this is only possible if 
	every vertex has degree $3\binom{n/2}2$. But now it follows that for every 
	$v\in V$ the hypergraph $H\sm v$ has exactly~$b(n-1)$ edges. 
	So if $n\ge 10$ the induction hypothesis informs us that the assumption of 
	Lemma~\ref{lem:even} is satisfied, meaning that $H$ is indeed isomorphic to $B_n$.
	In the remaining case~$n=8$ the same conclusion can still be drawn unless there is a vertex 
	$v\in V$ with $(H\sm v)\cong J_7$. But this would entail 
	$K_6^{(3)}\subseteq J_7\subseteq H$ and Fact~\ref{fact:8-46} would show
	$|E|<b(8)$, which contradicts the choice of $H$.
	
	\smallskip\goodbreak

	{\it \hskip1em Second Case: $n\ge 9$ is odd.} 

	\medskip

	For $K\subseteq V$ and $i\in \{0, 1, 2, 3\}$ let $e_i(K)$ denote the number of edges 
	of $H$ with exactly $i$ vertices in $K$. Clearly, we have 
	\begin{equation}\label{eq:K}
		e_0(K)+e_1(K)+e_2(K)+e_3(K)=|E|=b(n)
	\end{equation}    
	for every $K\subseteq V$.

	We need to know later that $H$ cannot contain a clique on five vertices 
	and the claim that follows prepares the proof of this fact. 
	\begin{clm}\label{clm:B6}
		If some six vertices of $H$ span at least $18$ edges, then they induce a copy 
		of $B_6$.
	\end{clm}
	
	\begin{proof}
		Let $K=\{v_1, \ldots, v_6\}\subseteq V$ span at least $18$ edges of $H$
		and suppose for the sake of contradiction that the subhypergraph of $H$ 
		induced by $K$ is not isomorphic to $B_6$. Arguing as in the second paragraph 
		of the proof of Lemma~\ref{lem:11-18} we may assume that $\{v_1, \ldots, v_5\}$ 
		induces a~$K^{(3)}_5$ in $H$.
		
		For $n\ge 13$ we have $n-6\ge 7$ and the induction hypothesis yields, 
		in particular, 
		\begin{equation}\label{eq:60}
			e_0(K)\le b(n-6)+\tfrac12(n-9)\,,
		\end{equation}
		where the additional term $\tfrac12(n-9)$ is actually not needed. 
		The reason for including it here is that for $n\in \{9, 11\}$ 
		it makes the right side equal to the trivial upper bound~$\binom{n-6}{3}$.
		Therefore~\eqref{eq:60} holds in all possible cases.
				
		Now consider the $6$-multigraph 
		$\seq{G}=(G_1, \ldots, G_6)$ with vertex set $V\sm K$, where for~$j\in [6]$ the edges 
		of $G_j$ are inherited from the link graph of~$v_j$. 
		Since $\{v_1, \ldots, v_5\}$ 
		is a clique in~$H$, three crossing pairs in $(G_1, \ldots, G_5)$ would give rise to a 
		Fano plane in $H$. Hence $e(G_1)+\ldots +e(G_5)\le f_5(n-6)$ and together with the 
		trivial bound $e(G_6)\le \binom{n-6}2$ we obtain 
		\begin{equation} \label{eq:61}
				e_1(K)=e(\seq{G}) \le f_5(n-6)+\binom{n-6}2\,.
		\end{equation}

		Moreover, Lemma~\ref{lem:11-18} shows that every vertex $v\in V\sm K$ can contribute 
		at most $10$ edges to~$e_2(K)$, wherefore 
		\begin{equation} \label{eq:62}
			e_2(K)\le 10(n-6)\,. 
		\end{equation}
		By plugging~\eqref{eq:60},~\eqref{eq:61},~\eqref{eq:62}, and the trivial upper 
		bound $e_3(K)\le\binom{6}{3}=20$ into~\eqref{eq:K}
		we arrive at the estimate
		\[ 
			b(n)\le \bigl(b(n-6)+\tfrac12(n-9)\bigr)+ f_5(n-6) +\binom{n-6}{2}+10(n-6) +20\,,
		\] 	
		which contradicts Corollary~\ref{conc:bf}\ref{it:b}. 
		This proves Claim~\ref{clm:B6}.
		\end{proof}
		
		\goodbreak
	
		\begin{clm} \label{clm:K5}
			$K^{(3)}_5\not\subseteq H$
		\end{clm}
		
		\begin{proof}
			Assume to the contrary that $K=\{v_1, \ldots, v_5\}\subseteq V$ induces a $K^{(3)}_5$
			in $H$. We contend that $e_0(K)\le b(n-5)$. For $n\ge 13$ this follows indeed 
			from the induction hypothesis, for~$n=9$ we just need to appeal to the trivial bound 
			$e_0(K)\le\binom{n-5}{3}=4=b(n-5)$ and for $n=11$ the desired estimate holds 
			in view of Claim~\ref{clm:B6}.
			
			A link multigraph argument similar to the one encountered in the 
			foregoing proof of~\eqref{eq:61} shows that $e_1(K)\le f_5(n-5)$. Owing to 
			Claim~\ref{clm:B6} every vertex in $V\sm K$ belongs to at most $7$ 
			edges contributing to $e_2(K)$, whence $e_2(K)\le 7(n-5)$. Combining all 
			these estimates and $e_3(K)=10$ with~\eqref{eq:K} we learn
			\[
				b(n)\le b(n-5)+f_5(n-5)+7(n-5)+10\,,
			\]
			which, however, contradicts Corollary~\ref{conc:bf}\ref{it:a}. 
			Thereby Claim~\ref{clm:K5} is proved.
		\end{proof}
		
		In order to conclude the proof of our main result we will now show 
		that $H$ satisfies the assumptions of Lemma~\ref{lem:odd}. Suppose to this 
		end that $K\subseteq V$ induces a tetrahedron in $H$. For $n\ge 11$ the induction 
		hypothesis gives $e_0(K)\le b(n-4)$ and for $n=9$ this estimate could only fail 
		if $V\sm K$ induces a $K_5^{(3)}$ in $H$, which would contradict Claim~\ref{clm:K5}.
		Thus we obtain
		\begin{equation} \label{eq:4}
			b(n)\le b(n-4)+f_4(n-4)+5(n-4)+4
		\end{equation}
		in the usual manner, where the factor $5$ in front of $(n-4)$ comes from 
		the absence of $5$-cliques in $H$. In view of~\eqref{eq:bn} and Theorem~\ref{thm:f4}
 		the right side equals 
		\begin{align*}
			&\frac18\bigl((n-4)^2-1\bigr)(n-6)+2\binom{n-4}{2}+\frac12\bigl((n-4)^2-1\bigr)
			+5(n-4)+4 \\
			=&
			\frac18(n^2-1)(n-2)=b(n)\,,
		\end{align*} 
		meaning that~\eqref{eq:4} actually holds with equality. In particular, this 
		yields 
		\begin{equation} \label{eq:40} 
			e_0(K)=b(n-4)
	    \end{equation} 
		and 
		\begin{equation*} \label{eq:42} 
			e_2(K)=5(n-4)\,.
	    \end{equation*} 

	    The latter equation proves immediately that $K$ obeys clause~\ref{it:oddii} 
	    from Lemma~\ref{lem:odd}. It remains to check that, similarly,~\eqref{eq:40}
	    leads to~\ref{it:oddi}, i.e., to $(H\sm K)\cong B_{n-4}$. 
	    For $n\ge 13$ this is indeed true due to the induction hypotheses. 
	    For $n=11$ we need to point out additionally that $H\sm K$ cannot be
	    isomorphic to $J_7$, as this hypergraph contains a copy of~$K_5^{(3)}$, whilst $H$ 
	    does not.
	    Finally, for $n=9$ the desired statement is a simple consequence of the fact 
	    that $B_5$, the five-clique with one edge removed, is the only hypergraph 
	    on~$5$ vertices with $b(5)=9$ edges.  
	    This concludes the proof of our main result.

\section*{Acknowledgement} We would like to thank Mikl\'os Simonovits for sending us 
a copy of~\cite{So76}, Zolt\'an F\"uredi~\cite{Fu} for further information regarding 
the history of the problem, and the referees for a careful reading of this article.


\begin{bibdiv}
\begin{biblist}

\bib{B}{book}{
   author={Bollob\'{a}s, B\'{e}la},
   title={Extremal graph theory},
   note={Reprint of the 1978 original},
   publisher={Dover Publications, Inc., Mineola, NY},
   date={2004},
   pages={xx+488},
   isbn={0-486-43596-2},
   review={\MR{2078877}},
}

\bib{Br83}{article}{
   author={Brown, W. G.},
   title={On an open problem of Paul Tur\'{a}n concerning $3$-graphs},
   conference={
      title={Studies in pure mathematics},
   },
   book={
      publisher={Birkh\"{a}user, Basel},
   },
   date={1983},
   pages={91--93},
   review={\MR{820213}},
}

\bib{DeFu00}{article}{
   author={De Caen, Dominique},
   author={F\"uredi, Zolt\'an},
   title={The maximum size of 3-uniform hypergraphs not containing a Fano
   plane},
   journal={J. Combin. Theory Ser. B},
   volume={78},
   date={2000},
   number={2},
   pages={274--276},
   issn={0095-8956},
   review={\MR{1750899}},
   doi={10.1006/jctb.1999.1938},
}


\bib{Er77}{article}{
   author={Erd{\H{o}}s, Paul},
   title={Paul Tur\'an, 1910--1976: his work in graph theory},
   journal={J. Graph Theory},
   volume={1},
   date={1977},
   number={2},
   pages={97--101},
   issn={0364-9024},
   review={\MR{0441657 (56 \#61)}},
}

\bib{EHSS}{article}{
   author={Erd{\H{o}}s, P.},
   author={Hajnal, A.},
   author={S{\'o}s, Vera T.},
   author={Szemer{\'e}di, E.},
   title={More results on Ramsey-Tur\'an type problems},
   journal={Combinatorica},
   volume={3},
   date={1983},
   number={1},
   pages={69--81},
   issn={0209-9683},
   review={\MR{716422}},
   doi={10.1007/BF02579342},
}

\bib{Fu}{unpublished}{
   author={F\"uredi, Zolt\'an},
   note={Personal communication},
}

\bib{FK02}{article}{
   author={F\"uredi, Zolt\'an},
   author={K\"undgen, Andr\'e},
   title={Tur\'an problems for integer-weighted graphs},
   journal={J. Graph Theory},
   volume={40},
   date={2002},
   number={4},
   pages={195--225},
   issn={0364-9024},
   review={\MR{1913847}},
   doi={10.1002/jgt.10012},
}

\bib{FuSi05}{article}{
   author={F\"uredi, Zolt\'an},
   author={Simonovits, Mikl\'os},
   title={Triple systems not containing a Fano configuration},
   journal={Combin. Probab. Comput.},
   volume={14},
   date={2005},
   number={4},
   pages={467--484},
   issn={0963-5483},
   review={\MR{2160414}},
   doi={10.1017/S0963548305006784},
}

\bib{KNS64}{article}{
   author={Katona, Gyula},
   author={Nemetz, Tibor},
   author={Simonovits, Mikl\'os},
   title={On a problem of Tur\'an in the theory of graphs},
   language={Hungarian, with Russian and English summaries},
   journal={Mat. Lapok},
   volume={15},
   date={1964},
   pages={228--238},
   issn={0025-519X},
   review={\MR{0172263}},
}

\bib{KeMu12}{article}{
   author={Keevash, Peter},
   author={Mubayi, Dhruv},
   title={The Tur\'an number of $F_{3,3}$},
   journal={Combin. Probab. Comput.},
   volume={21},
   date={2012},
   number={3},
   pages={451--456},
   issn={0963-5483},
   review={\MR{2912791}},
   doi={10.1017/S0963548311000678},
}

\bib{KeSu05}{article}{
   author={Keevash, Peter},
   author={Sudakov, Benny},
   title={The Tur\'an number of the Fano plane},
   journal={Combinatorica},
   volume={25},
   date={2005},
   number={5},
   pages={561--574},
   issn={0209-9683},
   review={\MR{2176425}},
   doi={10.1007/s00493-005-0034-2},
}

\bib{Ko82}{article}{
   author={Kostochka, A. V.},
   title={A class of constructions for Tur\'{a}n's $(3,\,4)$-problem},
   journal={Combinatorica},
   volume={2},
   date={1982},
   number={2},
   pages={187--192},
   issn={0209-9683},
   review={\MR{685045}},
   doi={10.1007/BF02579317},
}

\bib{LR-a}{article}{
	author={L\"uders, Clara Marie},
	author={Reiher, Christian}, 
	title={The Ramsey-Tur\'an problem for cliques}, 
	eprint={1709.03352},
	note={Israel Journal of Mathematics. To Appear},
}

\bib{LR-b}{article}{
	author={L\"uders, Clara Marie},
	author={Reiher, Christian}, 
	title={Weighted variants of the Andr\'asfai-Erd\H{o}s-S\'os Theorem}, 
	eprint={1710.09652},
	note={Submitted},
}

\bib{MuRo02}{article}{
   author={Mubayi, Dhruv},
   author={R\"odl, Vojt\^ech},
   title={On the Tur\'an number of triple systems},
   journal={J. Combin. Theory Ser. A},
   volume={100},
   date={2002},
   number={1},
   pages={136--152},
   issn={0097-3165},
   review={\MR{1932073}},
   doi={10.1006/jcta.2002.3284},
}

\bib{Pasch}{book}{
   author={Pasch, Moritz},
   title={Vorlesungen \"{u}ber neuere Geometrie},
   language={German},
   series={Teubner Studienb\"ucher Mathematik. [Teubner Mathematical
   Textbooks]},
   edition={Second edition},
   note={With supplementary material by Paul Bernays},
   publisher={B. G. Teubner, Leipzig und Berlin},
   date={1912},
   pages={iv+225},
   isbn={3-519-32020-7},
   review={\MR{1109913}},
   
}
		
\bib{Ra07}{article}{
   author={Razborov, Alexander A.},
   title={Flag algebras},
   journal={J. Symbolic Logic},
   volume={72},
   date={2007},
   number={4},
   pages={1239--1282},
   issn={0022-4812},
   review={\MR{2371204 (2008j:03040)}},
   doi={10.2178/jsl/1203350785},
}

\bib{Ra10}{article}{
   author={Razborov, Alexander A.},
   title={On 3-hypergraphs with forbidden 4-vertex configurations},
   journal={SIAM J. Discrete Math.},
   volume={24},
   date={2010},
   number={3},
   pages={946--963},
   issn={0895-4801},
   review={\MR{2680226 (2011k:05171)}},
   doi={10.1137/090747476},
}

\bib{RoSi95}{article}{
   author={R\"odl, Vojt\v ech},
   author={Sidorenko, Alexander},
   title={On the jumping constant conjecture for multigraphs},
   journal={J. Combin. Theory Ser. A},
   volume={69},
   date={1995},
   number={2},
   pages={347--357},
   issn={0097-3165},
   review={\MR{1313901}},
   doi={10.1016/0097-3165(95)90057-8},
}

\bib{Si68}{article}{
   author={Simonovits, M.},
   title={A method for solving extremal problems in graph theory, stability
   problems},
   conference={
      title={Theory of Graphs},
      address={Proc. Colloq., Tihany},
      date={1966},
   },
   book={
      publisher={Academic Press, New York},
   },
   date={1968},
   pages={279--319},
   review={\MR{0233735}},
}

\bib{So76}{article}{
   author={S\'os, Vera T.},
   title={Remarks on the connection of graph theory, finite geometry and
   block designs},
   language={English, with Italian summary},
   conference={
      title={Colloquio Internazionale sulle Teorie Combinatorie},
      address={Roma},
      date={1973},
   },
   book={
      publisher={Accad. Naz. Lincei, Rome},
   },
   date={1976},
   pages={223--233. Atti dei Convegni Lincei, No. 17},
   review={\MR{0543051}},
}
			
\bib{Turan}{article}{
   author={Tur{\'a}n, Paul},
   title={Eine Extremalaufgabe aus der Graphentheorie},
   language={Hungarian, with German summary},
   journal={Mat. Fiz. Lapok},
   volume={48},
   date={1941},
   pages={436--452},
   review={\MR{0018405}},
}

\bib{Tutte}{article}{
   author={Tutte, W. T.},
   title={The factorization of linear graphs},
   journal={J. London Math. Soc.},
   volume={22},
   date={1947},
   pages={107--111},
   issn={0024-6107},
   review={\MR{0023048}},
   doi={10.1112/jlms/s1-22.2.107},
}

\end{biblist}
\end{bibdiv}
\end{document}